\documentclass[reqno]{amsart}

\usepackage{amssymb,color}
\usepackage{amsmath}
\usepackage{amsfonts}
\usepackage{fouridx}
\usepackage[normalem]{ulem}
\usepackage{soul}
\usepackage[pdftex]{hyperref}
\usepackage{relsize}
\usepackage[utf8]{inputenc}
\usepackage{mathrsfs}
\usepackage[english]{babel}
\usepackage[totoc]{idxlayout}
\usepackage{fancyhdr}
\usepackage{paralist}
\usepackage{xcolor}
\usepackage[mathscr]{euscript}
\usepackage{mathtools}

\usepackage{todonotes}

\allowdisplaybreaks[3]
\newcommand\delc[1]{}

\newcommand\comcd[1]{}

\newcommand\del[1]{}
\newcommand\deln[1]{}
\newcommand\delr[1]{}

\newcommand\comad[1]{}

\newcommand\Greendel[1]{}
\newcommand\old[1]{}

\numberwithin{equation}{section}

\newtheorem{assumption}{Assumption}[section]

\def\old#1{}

\def\text#1{{\rm #1}}
\def\newold#1{}

\theoremstyle{plain}

\numberwithin{equation}{section}

\begin{document}

\title[Global solution of modified SH equation on Hilbert manifold]{On the Global Solution and Invariance of nonlinear Constrained Modified Swift-Hohenberg Equation on Hilbert Manifold}

\author{Saeed Ahmed}
\address{Department of Mathematics\\
Sukkur IBA University\\
Sindh Pakistan}
\email{saeed.msmaths21@iba-suk.edu.pk}

\author{Javed  Hussain}
\address{Department of Mathematics\\
Sukkur IBA University\\
Sindh Pakistan}
\email{javed.brohi@iba-suk.edu.pk}

\begin{abstract}
In this paper, we are interested in proving the existence and uniqueness of the local, local maximal, and global solutions of the equation projected on the Hilbert manifold. Furthermore, we show that, for any given initial data in the Hilbert manifold $\mathcal{M}$, the solution to this equation is also in the Hilbert manifold $\mathcal{M}$. Finally, we demonstrate that the solution to the equation is a gradient flow.
\end{abstract}
\date{\today}
\maketitle



\newtheorem{theorem}{Theorem}[section]
\newtheorem{lemma}[theorem]{Lemma}
\newtheorem{proposition}[theorem]{Proposition}
\newtheorem{corollary}[theorem]{Corollary}
\newtheorem{question}[theorem]{Question}

\theoremstyle{definition}
\newtheorem{definition}[theorem]{Definition}
\newtheorem{algorithm}[theorem]{Algorithm}
\newtheorem{conclusion}[theorem]{Conclusion}
\newtheorem{problem}[theorem]{Problem}

\theoremstyle{remark}
\newtheorem{remark}[theorem]{Remark}
\numberwithin{equation}{section}

\section{Introduction}

 This study investigates the following initial value problem related to the constrained modified Swift-Hohenberg equation.

\begin{eqnarray}{\label{main_Prb_1}}
        \frac{\partial u}{\partial t} &=&\pi _{u}(-\Delta^{2}u+2\Delta u -au - u^{2n-1}) \\
u(0) &=&u_{0}(x),~~~~~~~~~~\text{for}~~x \in  \mathcal{O}  \notag 
    \end{eqnarray}

 Where $u(x,t)$ evolves under \textbf{bihormonic operator} $(\Delta^{2}u)$, \textbf{negative diffusion} $(-2\Delta u)$, \textbf{a linear reaction term} $(au)$ and \textbf{a higher order non-linearity} $(u^{2n-1})$. Also,  $\mathcal{O} \in \mathcal{R}^{d}$ denotes a bounded domain with a smooth boundary $\partial \mathcal{O}$,  $n \in \mathbb{N} $  (or, in a general sense, an actual number such that $n>\frac{1}{2}$), and $u_{0} \in H_{0}^{1}(\mathcal{O}) \cap H^{2}(\mathcal{O})\cap \mathcal{M}$. \\

The expression on the right side of the equation $\Delta^{2}u-2\Delta u +au + u^{2n-1}$ comes from the following  modified Swift-Hohenberg equation. (\ref{dole}).  This equation was first studied by Doelman et al.\cite{Doelman} in 2003 for a pattern formation with two unbounded spatial directions.

\begin{align}{\label{dole}}
    u_{t}=-\alpha (1+ \Delta)^{2}u +\beta u-\gamma |\nabla u|^{2}-u^{3}
\end{align}
Here, $\alpha >0, \beta, and ~\gamma $ are constants. When $\gamma=0$, equation (\ref{dole}) becomes the usual Swift-Hohenberg equation. The extra term $\gamma |\nabla u|^{2}$  reminiscent of the Kuramoto-Sivashinsky equation, which arises in
the study of various pattern formation phenomena involving some kind of phase
turbulence or phase transition \cite{Kuramoto, Siivashinsky} breaks the symmetry $u \rightarrow -u$. In $1977$, J.B. Swift and P.C. Hohenberg introduced the Swift-Hohenberg equation about Rayleigh–Benard's convection. Then, it was found that this is a pivotal tool in studying various problems with pattern production, noise effects on bifurcations, defect dynamics, including pattern selection and spatiotemporal chaos can be modeled using this equation \cite{Nonlaopon, Ahmad, Amryeen, Wang1}. Additionally,  it has been studied in Taylor Taylor–Couette flows \cite{hohenberg1992effects}, \cite{Pomeau}, and in the study of lasers \cite{Lega}. There has been much research on the subject of MSHE as Wang, along with co-authors, has obtained statistical solutions for a nonautonomous modified Swift–Hohenberg equation in \cite{Wang}. Maria B. Kania has shown that semigroup generated by MSHE admits a global attractor in \cite{Maria}. In \cite{wangJintao},  Jintao Wang1, Chunqiu Li, Lu Yang, and Mo Jia have studied the long-time dynamical behaviors of 2D nonlocal stochastic Swift-Hohenberg equations with multiplicative noise. Yuncherl Choi, Taeyoung Ha, and Jongmin Han, in \cite{Yuncherl}, have studied the dynamical bifurcation of the modified Swift–Hohenberg equation endowed with an evenly periodic condition on the interval  $[-\lambda, \lambda]$. A fast and efficient numerical algorithm for Swift–Zhifeng Weng and others have also formulated the Hohenberg equation with a nonlocal nonlinearity in \cite{Zhifeng}.\\

To the best of our knowledge, a plethora of work has been conducted to study the modified Swift–Hohenberg equation, such as statistical solutions, global attractors, bifurcation analysis, and various numerical schemes. However, the abstract well-posedness for the projected versions of MSHE has not been studied yet. This study aims to fill this gap in literature. \\

This paper aims to present the well-posedness of the global solution to the deterministic constrained modified Swift Hohenberg equation(\ref{main_Prb_1}). The framework of this paper is as follows: by considering the approximated evolution equation and using the Banach fixed-point theorem, we will prove the existence and uniqueness of the solution. Furthermore, using Kartowski-Zorn Lemma, we will verify the local maximal solution to the approximated evolution equation. Then, we will show that if some constant bounds initial data $u_{0}$, say $R$, then the solution to the approximated evolution will be the same as the solution to the central equation (\ref{main_Prb_1}). In addition, we will show the invariance of manifold M; that is, if the initial data $u_{0}$ are in the Hilbert manifold, then the solution of the projected evolution equation will also be in the Hilbert manifold M. Finally, we will demonstrate the existence and uniqueness of the global solution along with the proof that the semiflow generated is a gradient flow.

\section{Notations and functional setting}

\addtocontents{toc}{\protect\setcounter{tocdepth}{1}}

\begin{assumption}{\label{Ass_2.2.2}}

    Assume that the spaces , $(\mathcal{E}, \|.\|) $, $(\mathcal{V}, \|.\|_{\mathcal{V}}) $  and $(\mathcal{H}, |.|) $ are Banach spaces and   

\begin{eqnarray*} 
\mathcal{H}&:=& L^2(\mathcal{O})\\
\mathcal{V} &:=& H_{0}^{1}(\mathcal{O}) \cap H^{2}(\mathcal{O})\\
\mathcal{E} &:=& D(A) = H_{0}^{1}(\mathcal{O}) \cap H^{2}(\mathcal{O}) \cap H^{4}(\mathcal{O}).
\end{eqnarray*} 
From (\ref{Self-Ad_op}), it is known that
 $A :  \mathcal{E} \rightarrow \mathcal{L}^{2}(\mathcal{O})$ is the self-adjoint operator, where
 $$ A = \Delta^{2}-2\Delta ~~~~~ \text{and} ~~~\mathcal{O} \subset \mathbb{R}^{d} $$
 and the following embeddings are dense and continuous: $$ \mathcal{E} \hookrightarrow \mathcal{V} \hookrightarrow  \mathcal{H}. $$

\end{assumption}

\subsection{Hilbert manifold, tangent space and orthogonal projection}

  \begin{definition}(\textbf{Hilbert Manifold})

    \cite{masiello1994variational} Assume that ${H}$ is a Hilbert space with inner product $\langle ~ \cdot, ~\cdot \rangle$, then the Hilbert manifold $\mathcal{M}$ is given as:
\begin{align*}
    \mathcal{M} =  \{~ u \in \mathcal{H}, &|u|_{\mathcal{H}}^{2}=1~\}     
\end{align*}   
\end{definition}

\begin{definition}{(\textbf{Tangent Space})}

\cite{masiello1994variational}, \cite{hussain2015analysis} At a point $u~ \in ~{H}$, Tangent space $T_{u}\mathcal{M}$ is given as:
\begin{align*}
    T_{u}\mathcal{M}= \{~ h: ~~\langle h, u \rangle = 0~ \},~~~~~~~h \in {H}
\end{align*}  
\end{definition}

The orthogonal projection of $h$ onto $u$, $ \pi_{u}~: ~ \mathcal{H} \longrightarrow T_{u}\mathcal{M}$ is defined as
\begin{equation}{\label{lemma_Tangent}}
    \pi_{u}(h)= h-\langle h, u \rangle ~u, ~~~~~~h \in \mathcal{H}
\end{equation}  

If $ u \in \mathcal{E} \cap \mathcal{M}$ then using (\ref{lemma_Tangent}),  the projection of $-\Delta^{2}u+2\Delta u +au + u^{2n-1}$ under the map $\pi_{u}$ can be found as:\\
$\pi _{u}(-\Delta^{2}u+2\Delta u -au - u^{2n-1})$
 \begin{eqnarray*}
  &=&-\Delta^{2}u+2\Delta u -au - u^{2n-1}
+\langle \Delta^{2}u-2\Delta u +au + u^{2n-1}, u \rangle ~u \notag \\
&=&-\Delta^{2}u+2 \Delta u -au - u^{2n-1} + \langle \Delta^{2}u, u \rangle ~u -2\langle \Delta u, u \rangle ~u \notag \\ & & ~+a\langle u, u  
\rangle ~u+\langle u^{2n-1}, u \rangle ~u \notag \\
&=&-\Delta^{2}u+2 \Delta u -au - u^{2n-1} + \langle \Delta u,  \Delta u \rangle ~u -2\langle - \nabla u, \nabla u \rangle ~u \notag \\ & & ~+a\langle u, u  
\rangle ~u+\langle u^{2n-1}, u \rangle ~u \notag \\
&=&-\Delta^{2}u+2 \Delta u -au - u^{2n-1} + \| \Delta u\|^{2}_{{L}^{2}(\mathcal{O})} ~u + 2\| \nabla u\|^{2}_{{L}^{2}(\mathcal{O})} ~u \notag \\ & & ~ +au+\| u\|^{2n}_{{L}^{2n}(\mathcal{O})} u \notag \\ 
&=&-\Delta^{2}u+2 \Delta u  + \|  u\|^{2}_{{H}^{2}_{0}} ~u + 2\|  u\|^{2}_{{H}^{1}_{0}} ~u \notag  ~ +\| u\|^{2n}_{{L}^{2n}} u- u^{2n-1}  \\
\end{eqnarray*}    
Therefore, the projection of $-\Delta^{2}u+2\Delta u +au + u^{2n-1}$ under map  $\pi_{u}$ is
\begin{eqnarray} {\label{Projection_U}}
      &&\pi _{u}(-\Delta^{2}u+2\Delta u-au - u^{2n-1}) \notag \\ &=&  -  \Delta^{2}u+2 \Delta u  + \|  u\|^{2}_{{H}^{2}_{0}} ~u + 2\|    u\|^{2}_{{H}^{1}_{0}} ~u  +\| u\|^{2n}_{{L}^{2n}} u- u^{2n-1} 
\end{eqnarray}

\addtocontents{toc}{\protect\setcounter{tocdepth}{1}}
    \subsection{Main Problem}

 Assume that $\mathcal{E}$, $\mathcal{H}$ and $\mathcal{V}$ are Hilbert spaces and satisfy Assumption (\ref{Ass_2.2.2}). The following problem will be discussed in this section.
    \begin{eqnarray}{\label{main_Prb}}
        \frac{\partial u}{\partial t} &=&\pi _{u}(-\Delta^{2}u+2\Delta u -au - u^{2n-1}) \\
u(0) &=&u_{0}.  \notag 
    \end{eqnarray}
where   $n \in \mathbb{N} $  (or, in general, an actual number such that $n>\frac{1}{2}$) and $u_{0} \in \mathcal{V} \cap \mathcal{M}$. We want to prove that problem (\ref{main_Prb}) has a unique global solution in the Hilbert manifold. Before establishing this, we will discuss the well-posedness of the Abstract Problem in the next section.  
\section{Local maximal solution to Abstract Problem}
 Consider the following problem of parabolic type:\\
 
  \begin{eqnarray} {\label{Ab_prb}}
     \frac{\partial u}{\partial t} &=&-A u(t)
+ F(u(t)), ~~~~~~t \geq 0\notag  \\
u(0) &=&u_{0}.  \notag
  \end{eqnarray}

 Where $A$ is the self-adjoint operator and $F(u(t))$ is a local Lipschitz map from $\mathcal{V}$ to $\mathcal{H}$. 
The next theorem proves that if the abstract spaces $\mathcal{E}$, $\mathcal{V}$ and $\mathcal{H}$ satisfy (\ref{Ass_2.2.2}), then problem (\ref{Ab_prb}) admits a unique local solution, a local maximal solution, and a global solution.
\addtocontents{toc}{\protect\setcounter{tocdepth}{1}}
\subsection{Solution space}
Assume that  $\mathcal{E}$, $\mathcal{V}$ and $\mathcal{H}$ are abstract Banach spaces, then we denote 
\begin{equation*}
   X_{T}:=L^{2}\left( 0,T;\mathcal{E}\right) \cap C\left( \left[ 0,T\right] ;\mathcal{V}\right) ,
\end{equation*} 
then it can be easily proven that  $\left( X_{T},\left\vert \cdot\right\vert _{X_{T}}\right) $  is the Banach space with the following norm, 
\begin{equation*}
\left\vert u\right\vert _{X_{T}}^{2}=\underset{p\in \lbrack 0,T]}{\sup }\left\Vert u(p)\right\Vert ^{2}
+\int_{0}^{T}\left\vert u(p)\right\vert_{\mathcal{E}}^{2}dp,
\end{equation*} 

The following assumptions are instrumental and will be implemented throughout this chapter.

\begin{assumption}

 Assume that for the spaces $\mathcal{E} \subset\mathcal{V} \subset \mathcal{H}$,  the assumption (\ref{Ass_2.2.2}) is held and let $ \{S(t), 0 \leq t < \infty \}$ is an analytic Semi-group of the bounded and linear operators on $\mathcal{H}$  such that $ \exists $  $c_{1}$ and $c_{2}$ and \\
\\
i) For each $T>0$ and $f \in L^{2}\left( 0,T;\mathcal{E}\right)$, map $u=S*f $ is defined by
\begin{eqnarray*}
    u(t) = \int^{T}_{0} {S(t-p)f(p)}~dp,~~~~~ t \in [0,T]
\end{eqnarray*} 

And  $u(t) \in X_{T}$  such that 
\begin{eqnarray*}
  |u|_{X_{T}} \leq c_{1} |f|_{L^{2}\left( 0,T;\mathcal{E}\right)}
\end{eqnarray*}

Where $S* :  L^{2}\left( 0,T;\mathcal{E}\right)  \rightarrow X_{T}$. \\
\\
ii) For all $T>0$ and $u_{0} \in \mathcal{V}$, map $u=S(\cdot)u_{0} $ is given by
\begin{eqnarray*}
    u(t) = S(t) u_{0},~~~~~ t \in [0,T] 
\end{eqnarray*} 
is in $X_{T}$ and 
\begin{eqnarray*}
  |u|_{X_{T}} \leq c_{2} \|u_{0}\|_{\mathcal{V}}
\end{eqnarray*}
\end{assumption}

 Next, we define the auxiliary function.
 Let $\theta : [0,1] \longrightarrow R^{+}$ be a function with compact support and is a non-increasing function such that 
\begin{align}{\label{truncted}}
    \begin{cases}   
\theta(x) = 1, ~~ \text{iff} ~~x \in [0,1] \\
\theta(x) = 0, ~~ \text{iff}~~ x \in [2,\infty] \\
\inf_{ x \in R^{+}} \theta'(x)  \geq -1 
\end{cases} 
\end{align}
 For $m \geq 1$,  ~~~~~~       $ \theta_{m}(\cdot) = \theta (\frac{\cdot}{m})$

\begin{lemma}
    (\cite{brzezniak2014stochastic}, page 57) Assume that $g: \mathcal{R}^{+} \longrightarrow \mathcal{R}^{+}$ is a non-decreasing function, for every $x_{1},x_{2}\in R$ 
\begin{eqnarray} {\label{tr_2}}
    \theta_{n}(x_{1}) g(x_{1}) \leq g(2m) ~~~~\text{and}~~~~
    |\theta_{m}(x_{1})-\theta_{m}(x_{2})|\leq \frac{1}{m}|x_{1}-x_{2}|
\end{eqnarray}

\end{lemma}

\begin{definition}
    For $u_{0} \in \mathcal{V}$ the function $u: [0,T_{1}) \longrightarrow \mathcal{V}$ is the \textbf{local mild solution}  to the abstract problem (\ref{Ab_prb})if: \\
\\
 i) $\forall t \in [0,T_{1})  $, and $u|_{[0,t)} \in X_{t}$ \\
 \\
 ii) $\forall t \in [0,T_{1})  $
 $$  u(t) =S(t)u_{0} + \int^{t}_{0} S(t-p)F(u(p)) ~dp$$

 Moreover, a local mild solution is a \textbf{maximal local solution} iff for each local mild solution $\left( \widetilde{u(t)}, t \in [~0,  \widetilde{T})~ \right)$ the following conditions hold:\\
 \\
i) $\widetilde{T} \geq T $ \\
\\
ii)  $ \widetilde{u(t)} \vert_{[0,T)}  = u(t) ~~~ \text{and} ~~\widetilde{T}=T $ \\
\\
A local mild solution is referred as \textbf{global solution} iff $T= \infty$
\end{definition}

\begin{proposition}{\label{Pro.1}}
    \cite{hussain2015analysis} Assume that $\mathcal{E}$, $\mathcal{V}$ and $\mathcal{H}$ satisfy the assumption (\ref{Ass_2.2.2}) and $A$ is the abstract self-adjoint operator with $ F: \mathcal{V} \longrightarrow \mathcal{H}$ is locally Lipschitz function satisfying symmetric condition, that is, 
$$ |F(u_{1})-F(u_{2})|_{\mathcal{H}} \leq  \|u_{1}-u_{2}\|_{\mathcal{V}} \mathcal{G}(\|u_{1}\|, \|u_{2}\|)$$ \\
where $ \mathcal{G}: R^{+}\times R^{+} \longrightarrow R^{+}$ is a function satisfying the symmetric condition such that
\begin{eqnarray}
    |\mathcal{G}(m,n) | \leq C_{L}  ~~~ \text{for all} ~~m, n \in [0,L]
\end{eqnarray}  
Assume that $\theta_{n}$ be the function defined in (\ref{truncted}) and (\ref{tr_2})then the map $ \xi:X_{T} \longrightarrow L^{2}\left( 0,T;\mathcal{E}\right)$  defined as:\\
\begin{eqnarray}
    \xi[u\left( t\right)]= \theta_{m}\left( |u^{m}|_{X_{p}}\right) F \left( u(t) \right) 
\end{eqnarray}
is globally Lipschitz function that is\\
for any $u_{1}$ and $u_{2}$ we have\\
\begin{eqnarray}
    |\xi (u_{1})-\xi(u_{2})|_{L^{2}\left( 0,T;\mathcal{E}\right)}  \leq K_{m}|u_{1}-u_{2}|T^{\frac{1}{2}}
\end{eqnarray}
where $ K_{n} = \sqrt{\mathcal{G}^{2}(2n,2n)}+ 2n^{2}|\mathcal{G}(2n,0)|$
\end{proposition}

\begin{proposition}{\label{Pro.2}}
    \cite{hussain2015analysis}
    Assume that  $\mathcal{E}$, $\mathcal{V}$ and $\mathcal{H}$ follow assumption (\ref{Ass_2.2.2}). And consider the map $ \Phi:X_{T} \longrightarrow X_{T}$ defined by 
\begin{eqnarray}
    \Phi(u) = Su_{0}+ S* \xi(u) 
\end{eqnarray}

where  $ \xi$ is the map defined in Proposition (\ref{Pro.1}) with $u_{0}\in \mathcal{V}$. Then, for each $t$ there exists $T_{1}$ such that $ \forall t \in [0,T_{1})$, ~ $\Phi(u)$ is a strict contraction. So, by using Banach fixed Point theorem ~$ \forall t \in [0,T_{1})$ there is a unique $u^{n} \in X_{T}$ such that  ~~ $ \Phi(u^{n}) = u^{n} $
\end{proposition}

\begin{remark}
    By the propositions (\ref{Pro.1}) and (\ref{Pro.2}), we can say that there is a unique local mild solution to the  following approximated evolution equation 
\begin{eqnarray} {\label{appr_ev}}
    u^{n}(t) &=&S(t)u_{0} + \int^{t}_{0} \theta_{n}\left( |u^{n}|_{X_{p}}\right)S(t-p)F(u^{n}(p)) ~dp~~~~t \in[0,T] \notag \\
    u^{n}&=& u_{0}, ~~~~~~~~~~ u_{0} \in \mathcal{V}
\end{eqnarray}

Has a unique local mild solution.\\
\\ 
In the following subsection, the lemma shows the existence of a local maximal solution to equation (\ref{appr_ev}) using the Zorn Lemma, given that $\|u_{0}\|_{\mathcal{V}} < \infty$. 
\end{remark}

\begin{lemma}{\label{zorn}}
    \cite{brzezniak2014stochastic} Assume $ \mathcal{T} $ is  set of all local mild solutions and  for any two solutions $u_{1}$  and $u_{2}$ in $ \mathcal{T} $ defined on the intervals $[0,\tau) $ and $[0,\tau') $ respectively with the order $\prec $ defined as:
 \begin{eqnarray}
     u_{1} \prec u_{2}~~~~~\text{iff} ~~~~~\tau \leq \tau' ~~\text{and} ~~~~ u_{2}|_{[0,\tau)} = u_{1}
 \end{eqnarray}

then $ \mathcal{T} $  contains a maximal element.
\end{lemma}

\begin{remark}
   A local maximal solution to the problem (\ref{appr_ev}) can be obtained as the consequence of lemma (\ref{zorn})
\end{remark}

\begin{proposition}{\label{proposition for unique local solution}}

\cite{hussain2015analysis} For a given $K>0$ there is a $T'(K)$ such that for each $u_{0} \in \mathcal{V}$ and $\|u_{0}\|_{\mathcal{V}} \leq K $, there is the unique local solution $u : [0,T') \longrightarrow \mathcal{V}$ to the problem (\ref{Ab_prb}).

    \end{proposition}

\begin{proposition}{\label{proposition_for_global_solution}} 
\cite{hussain2015analysis} If $u$ is the local  mild solution of the abstract problem (\ref{Ab_prb}), where $t \in [0,T)$ and if there is a constant $K$ such that 
\begin{align*}
    \sup_{t \in [0,T)}{\|u(t)\|_{\mathcal{V}}} \leq K 
\end{align*}
then $t= \infty$. In other words, $u$ is the global solution.

\section{ Local maximal solution to the main problem} 

Before demonstrating the existence of a local maximal solution to the main problem (\ref{main_Prb}), we discuss some critical results in the following subsections.

\end{proposition}

\begin{lemma}{\label{lemma_abt_liptz}}    
Let $\mathcal{E}$, $\mathcal{V}$ and $\mathcal{H}$ satisfy the assumption (\ref{Ass_2.2.2}), and $ F :  \mathcal{V}\longrightarrow \mathcal{H}$ be a map defined as 
\begin{eqnarray*}
    F(u)=\|  u\|^{2}_{{H}^{2}_{0}} u + 2\|u\|^{2}_{{H}^{1}_{0}} u  +\| u\|^{2n}_{{L}^{2n}} u- u^{2n-1} 
\end{eqnarray*}
Then $F$ is locally Lipschitz that is 
\begin{equation*}
    \|F(u_{1}-F(u_{2})\|_{\mathcal{H}} \leq \mathcal{G} ( \|u_{1}\|_{\mathcal{V}},\|u_{2}\|_{\mathcal{V}}) \|u_{1}-u_{2}\|_{\mathcal{V}}, \quad u_{1}, u_{2}\in \mathcal{V}
\end{equation*} 
 
where $ \mathcal{G}:   [0, \infty) \times [0, \infty) \longrightarrow [0, \infty)$ denotes a bounded linear polynomial map. That is, 
\begin{eqnarray*}
    \mathcal{G} \left(m,n\right) &=&2C \left(m^{2}+n^{2}+mn\right) \notag \\ &&+C_{n} \left[  \left( \frac{2n-1}{2}\right) \left( m^{2n-1}+n^{2n-1}\right) (m+n)  + \left( m^{2n}+n^{2n}\right) + \left( 1+m^{2}+n^{2}\right)^{\frac{1}{3}} \right]
\end{eqnarray*}
\end{lemma} 

\begin{proof}
Let $  F(u)=\|  u\|^{2}_{{H}^{2}_{0}} u + 2\|    u\|^{2}_{{H}^{1}_{0}} u  +\| u\|^{2n}_{{L}^{2n}} u- u^{2n-1} = \sum _{i=1}^{4} F_{i}(u)$, i.e. set\\
\begin{equation*}
F_{1}(u) := \|  u\|^{2}_{{H}^{2}_{0}} u, \quad F_{2}(u) :=  2\|    u\|^{2}_{{H}^{1}_{0}} u,\quad   F_{3}(u) := \| u\|^{2n}_{{L}^{2n}} u,\quad F_{4}(u) := - u^{2n-1}.
\end{equation*}

 We consider each $F_{i}'s$ in the estimation. 
 Let us begin with $F_{1}$, assume the $u_{1}$ and $u_{2}$ are fixed elements in $\mathcal{V}$ then 
 \begin{eqnarray*}{}
     \left|F_{1}(u_{1})-F_{1}(u_{2})\right|_{\mathcal{H}} &=& \left| \|  u_{1}\|^{2}_{{H}^{2}_{0}} u_{1} - \|  u_{2}\|^{2}_{{H}^{2}_{0}} u_{2} \right|_{\mathcal{H}} \notag \\  
&=& \left| \|  u_{1}\|^{2}_{{H}^{2}_{0}} u_{1} -\|  u_{1}\|^{2}_{{H}^{2}_{0}} u_{2}+ \|  u_{1}\|^{2}_{{H}^{2}_{0}} u_{2} - \|  u_{2}\|^{2}_{{H}^{2}_{0}} u_{2} \right|_{\mathcal{H}} \notag \\  
 &\leq& \|  u_{1}\|^{2}_{{H}^{2}_{0}} \left|u_{1}-u_{2}\right|_{\mathcal{H}} + \left( \|  u_{1}\|^{2}_{{H}^{2}_{0}} -\|  u_{2}\|^{2}_{{H}^{2}_{0}} \right)\left|u_{2}\right|_{\mathcal{H}}  \notag \\
     &\leq&  \left(\|  u_{1}\|^{2}_{{H}^{2}_{0}} + \|  u_{2}\|^{2}_{{H}^{2}_{0}}\right) \left|u_{1}-u_{2}\right|_{\mathcal{H}} \\ && + \left( \|  u_{1}\|_{{H}^{2}_{0}} +\|  u_{2}\|_{{H}^{2}_{0}} \right) \|u_{1}-u_{2}\|_{{H}^{2}_{0}} \left(\left|u_{1}\right|_{\mathcal{H}}+\left|u_{2}\right|_{\mathcal{H}}\right)  \notag \\
     \end{eqnarray*}
By using $ {H}^{2}_{0} \hookrightarrow \mathcal{H}$ is continuous,  there is a constant  $C_{1}>0$ such that  $ |u|_{\mathcal{H}} \leq C _{1}\|u\|_{{H}^{2}_{0}}$,  and it follows that
\begin{eqnarray*}
     \left|F_{1}(u_{1})-F_{1}(u_{2})\right|_{\mathcal{H}} &\leq& C_{1} \left( \left(\|  u_{1}\|^{2}_{{H}^{2}_{0}} + \|  u_{2}\|^{2}_{{H}^{2}_{0}}\right)   + \left( \|  u_{1}\|_{{H}^{2}_{0}} +\|  u_{2}\|_{{H}^{2}_{0}} \right)^{2}\right) \|u_{1}-u_{2}\|_{{H}^{2}_{0}} 
\end{eqnarray*}
In addition, by using  $ \mathcal{V} \hookrightarrow {H}^{2}_{0} $ is continuous,  there is a constant  $C_{2}>0$ such that  $ |u|_{{H}^{2}_{0} } \leq C_{2} \|u\|_{\mathcal{V}}$,  and it follows that
\begin{eqnarray*}
     \left|F_{1}(u_{1})-F_{1}(u_{2})\right|_{\mathcal{H}} &\leq& C_{1} C_{2}^{3}\left( \left(\|  u_{1}\|^{2}_{\mathcal{V}} + \|  u_{2}\|^{2}_{\mathcal{V}}\right)   + \left( \|  u_{1}\|_{\mathcal{V}} +\|  u_{2}\|_{\mathcal{V}} \right)^{2}\right) \|u_{1}-u_{2}\|_{\mathcal{V}} \notag \\ 
     &=& 2C_{1} C_{2}^{3}\left( \|  u_{1}\|^{2}_{\mathcal{V}} + \|  u_{2}\|^{2}_{\mathcal{V}}   +  \|  u_{1}\|_{\mathcal{V}} \|  u_{2}\|_{\mathcal{V}} \right) \|u_{1}-u_{2}\|_{\mathcal{V}}
\end{eqnarray*}
Therefore, it was found that:

\begin{eqnarray}{\label{F_{1}}}
    \left|F_{1}(u_{1})-F_{1}(u_{2})\right|_{\mathcal{H}} &\leq& 2C_{1} C_{2}^{3}\left( \|  u_{1}\|^{2}_{\mathcal{V}} + \|  u_{2}\|^{2}_{\mathcal{V}}   +  \|  u_{1}\|_{\mathcal{V}} \|  u_{2}\|_{\mathcal{V}} \right) \|u_{1}-u_{2}\|_{\mathcal{V}}
\end{eqnarray}

Now, consider function $F_{2}$. Assume that $u_{1}$ and $u_{2}$  are two fixed element in $\mathcal{V}$

\begin{eqnarray*}
     \left|F_{2}(u_{1})-F_{2}(u_{2})\right|_{\mathcal{H}} &=& 2 \left| \|  u_{1}\|^{2}_{{H}^{1}_{0}} u_{1} - \|  u_{2}\|^{2}_{{H}^{1}_{0}} u_{2} \right|_{\mathcal{H}} \notag \\  
&=& 2\left| \|  u_{1}\|^{2}_{{H}^{1}_{0}} u_{1} -\|  u_{1}\|^{2}_{{H}^{1}_{0}} u_{2}+ \|  u_{1}\|^{2}_{{H}^{1}_{0}} u_{2} - \|  u_{2}\|^{2}_{{H}^{1}_{0}} u_{2} \right|_{\mathcal{H}} \notag \\  
 &\leq& 2\|  u_{1}\|^{2}_{{H}^{1}_{0}} \left|u_{1}-u_{2}\right|_{\mathcal{H}} + \left( \|  u_{1}\|^{2}_{{H}^{1}_{0}} -\|  u_{2}\|^{2}_{{H}^{1}_{0}} \right)\left|u_{2}\right|_{\mathcal{H}}  \notag \\
     &\leq& 2 \left(\|  u_{1}\|^{2}_{{H}^{2}_{0}} + \|  u_{2}\|^{2}_{{H}^{1}_{0}}\right) \left|u_{1}-u_{2}\right|_{\mathcal{H}} \\ && + \left( \|  u_{1}\|_{{H}^{1}_{0}} +\|  u_{2}\|_{{H}^{1}_{0}} \right) \|u_{1}-u_{2}\|_{{H}^{1}_{0}} \left(\left|u_{1}\right|_{\mathcal{H}}+\left|u_{2}\right|_{\mathcal{H}}\right)  \notag \\
     \end{eqnarray*}
By using $ {H}^{1}_{0} \hookrightarrow \mathcal{H}$ is continuous,  there is a constant  $C_{3}>0$ such that  $ |u|_{\mathcal{H}} \leq C _{3}\|u\|_{{H}^{1}_{0}}$,  and it follows that
\begin{eqnarray*}
     \left|F_{2}(u_{1})-F_{2}(u_{2})\right|_{\mathcal{H}} &\leq& 2C_{3} \left( \left(\|  u_{1}\|^{2}_{{H}^{1}_{0}} + \|  u_{2}\|^{2}_{{H}^{1}_{0}}\right)   + \left( \|  u_{1}\|_{{H}^{1}_{0}} +\|  u_{2}\|_{{H}^{1}_{0}} \right)^{2}\right) \|u_{1}-u_{2}\|_{{H}^{1}_{0}} 
\end{eqnarray*}

In addition, by using $ \mathcal{V} \hookrightarrow {H}^{1}_{0} $ is continuous,  there is a constant  $C_{4}>0$ such that  $ |u|_{{H}^{1}_{0} } \leq C_{4} \|u\|_{\mathcal{V}}$,  and it follows that

\begin{eqnarray*}
     \left|F_{2}(u_{1})-F_{2}(u_{2})\right|_{\mathcal{H}} &\leq& 2C_{3} C_{4}^{3}\left( \left(\|  u_{1}\|^{2}_{\mathcal{V}} + \|  u_{2}\|^{2}_{\mathcal{V}}\right)   + \left( \|  u_{1}\|_{\mathcal{V}} +\|  u_{2}\|_{\mathcal{V}} \right)^{2}\right) \|u_{1}-u_{2}\|_{\mathcal{V}} \notag \\ 
     &=& 4C_{3} C_{4}^{3}\left( \|  u_{1}\|^{2}_{\mathcal{V}} + \|  u_{2}\|^{2}_{\mathcal{V}}   +  \|  u_{1}\|_{\mathcal{V}} \|  u_{2}\|_{\mathcal{V}} \right) \|u_{1}-u_{2}\|_{\mathcal{V}}.
\end{eqnarray*}
Therefore, it is found that:
\begin{eqnarray}{\label{F_{2}}}
    \left|F_{2}(u_{1})-F_{2}(u_{2})\right|_{\mathcal{H}} &\leq& 4C_{3} C_{4}^{3}\left( \|  u_{1}\|^{2}_{\mathcal{V}} + \|  u_{2}\|^{2}_{\mathcal{V}}   +  \|  u_{1}\|_{\mathcal{V}} \|  u_{2}\|_{\mathcal{V}} \right) \|u_{1}-u_{2}\|_{\mathcal{V}}
\end{eqnarray}
Consider the function $F_{3}(u) = \| u\|^{2n}_{{L}^{2n}} u$,\\
$  \left|F_{3}(u_{1})-F_{3}(u_{2})\right|_{\mathcal{H}} $
  \begin{eqnarray*}
   &=&  \left| \|  u_{1}\|^{2n}_{{L}^{2n}} u_{1} - \|  u_{2}\|^{2n}_{{L}^{2n}} u_{2} \right|_{\mathcal{H}} \notag \\  
&=& \left| \|  u_{1}\|^{2n}_{{L}^{2n}} u_{1} -\|  u_{1}\|^{2n}_{{L}^{2n}} u_{2}+ \|  u_{1}\|^{2n}_{{L}^{2n}} u_{2} - \|  u_{2}\|^{2n}_{{L}^{2n}} u_{2} \right|_{\mathcal{H}} \notag \\   
 &\leq&  \left( \|  u_{1}\|^{2n}_{{L}^{2n}} -\|  u_{2}\|^{2n}_{{L}^{2n}} \right)\left|u_{1}\right|_{\mathcal{H}} + \|  u_{2}\|^{2n}_{{L}^{2n}} \left|u_{1}-u_{2}\right|_{\mathcal{H}}  \notag \\
     &\leq&  \left( \frac{2n-1}{2}\right) \left( \|  u_{1}\|^{2n-1}_{{L}^{2n}} -\|  u_{2}\|^{2n-1}_{{L}^{2n}} \right)\left|u_{1}\right|_{\mathcal{H}} + \|  u_{2}\|^{2n}_{{L}^{2n}} \left|u_{1}-u_{2}\right|_{\mathcal{H}}  \notag
     \end{eqnarray*}

As we know that; $ \mathcal{V} \hookrightarrow {L}^{2n-1} $, $ \mathcal{V} \hookrightarrow \mathcal{H} $ and $ \mathcal{V} \hookrightarrow {L}^{2n} $ are continuous , it follows that:\\
$ \left|F_{3}(u_{1})-F_{3}(u_{2})\right|_{\mathcal{H}}$
\begin{eqnarray*} 
&\leq& C^{n+1} \left( \frac{2n-1}{2}\right) \left( \|  u_{1}\|^{2n-1}_{\mathcal{V}} -\|  u_{2}\|^{2n-1}_{\mathcal{V}} \right)\|u_{1}\|_{\mathcal{V}} + C^{n+1}\|  u_{2}\|^{2n}_{\mathcal{V}} \|u_{1}-u_{2}\|_{\mathcal{V}} \\ 
 &\leq& C^{n+1} \left( \frac{2n-1}{2}\right) \left( \|  u_{1}\|^{2n-1}_{\mathcal{V}} +\|  u_{2}\|^{2n-1}_{\mathcal{V}} \right)\left(\|u_{1}\|_{\mathcal{V}} 
 +  \|u_{2}\|_{\mathcal{V}}\right)\\
 &\leq& C^{n+1} \left( \frac{2n-1}{2}\right)\left(\|u_{1}\|^{2n}_{\mathcal{V}} + \|  u_{2}\|^{2n}_{\mathcal{V}}\right) \|u_{1}-u_{2}\|_{\mathcal{V}}.  
\end{eqnarray*} 

We have:
\begin{eqnarray}{\label{F_{3}}}
    \left|F_{3}(u_{1})-F_{3}(u_{2})\right|_{\mathcal{H}}&\leq& C^{n+1} \left( \frac{2n-1}{2}\right)
 \left(\|u_{1}\|^{2n}_{\mathcal{V}} + \|  u_{2}\|^{2n}_{\mathcal{V}}\right) \|u_{1}-u_{2}\|_{\mathcal{V}}.
\end{eqnarray}

Finally, to prove the estimation for $F_{4}(u)=-u^{2n-1}$, the following inequality is required:

\begin{align}{\label{inq}}
\left| |u_{1}|_{\mathcal{H}}^{2n-2} u_{1} - |u_{2}|_{\mathcal{H}}^{2n-2} u_{2} \right|_{\mathcal{H}} \leq c \left( |u_{1}|_{\mathcal{H}}^{2n-2} + |u_{2}|_{\mathcal{H}}^{2n-2} \right) |u_{1} - u_{2}|_{\mathcal{H}}
\end{align}

We take steps to prove the inequality (\ref{inq}). First, we suppose that $|u_{1}|_{\mathcal{H}}, ~|u_{2}|_{\mathcal{H}} \leq 1$ then, the differentiation at zero of the function $u \rightarrow |u|^{p} $, where $p>1$, or one-sided differentiability at zero for $ p = 1$ gives the following equation:

\begin{align}{\label{sup}}
    \sup_{\substack{u_{1} \neq u_{2}, \\ |u_{1}|_{\mathcal{H}}, |u_{2}|_{\mathcal{H}} \leq 1}} \frac{\left| |u_{1}|_{\mathcal{H}}^{2n-2} u_{1} - |u_{2}|_{\mathcal{H}}^{2n-2} u_{2}\right|_{\mathcal{H}}}{|u_{1} - u_{2}|_{\mathcal{H}}} =: C_0 < \infty
\end{align}

For the case  $|u_{1}|_{\mathcal{H}} ~\text{ or } ~|u_{2}|_{\mathcal{H}} > 1$, we can continue as follows:

\begin{align*}
     \frac{\left| |u_{1}|_{\mathcal{H}}^{2n-2} u_{1} - |u_{2}|_{\mathcal{H}}^{2n-2} u_{2}\right|_{\mathcal{H}}}{|u_{1} - u_{2}|_{\mathcal{H}}} = 
\left(L_{1}+L_{2}\right)^{2n-2} \frac{\left| \left|u'_{1}\right|_{\mathcal{H}}^{2n-2} u'_{1} - \left|u'_{2} \right|_{\mathcal{H}}^{2n-2} u'_{2}\right|_{\mathcal{H}}}{ \left|u'_{1} - u'_{2}\right|_{\mathcal{H}}}
\end{align*}

Where $L_{1}=|u_{1}|_{\mathcal{H}}$, $L_{2}=|u_{2}|_{\mathcal{H}}$, $u'_{1}= \frac{u_{1}}{L_{1}+L_{2}}$,  and $u'_{2}= \frac{u_{2}}{L_{1}+L_{2}}$.

Using (\ref{sup}), it follows that

\begin{align*}
     \frac{\left| |u_{1}|_{\mathcal{H}}^{2n-2} u_{1} - |u_{2}|_{\mathcal{H}}^{2n-2} u_{2}\right|_{\mathcal{H}}}{|u_{1} - u_{2}|_{\mathcal{H}}} &\leq C_{0}
\left(L_{1}+L_{2}\right)^{2n-2} = C_{0}
\left(|u_{1}|_{\mathcal{H}}+|u_{2}|_{\mathcal{H}}\right)^{2n-2}
\leq C_{n}
\left(|u_{1}|_{\mathcal{H}}^{2n-2}+|u_{2}|_{\mathcal{H}}^{2n-2}\right)
\end{align*}

Therefore, we used the claim in (\ref{inq}). Our final estimation can be deduced using the Hölder inequality as follows:

\begin{eqnarray*}
     \left|F_{4}(u_{1})-F_{4}(u_{2})\right|_{\mathcal{H}}^{2} &=& \left|-u_{1}^{2n-1}+u_{2}^{2n-1}\right|_{\mathcal{H}}^{2} \leq  \left|u_{1}^{2n-2}-u_{2}^{2n-2}\right|_{\mathcal{H}}^{2} \left| u_{1}-u_{2}\right|_{\mathcal{H}}^{2}\notag \\
 &\leq& C^{2}_{n} \int_{\mathcal{O}} \left| u_{1}(s) - u_{2}(s) \right|_{\mathcal{H}}^{2} \left( \left| u_{1}(s) \right|_{\mathcal{H}}^{2n-2} + \left| u_{2}(s) \right|_{\mathcal{H}}^{2n-2} \right)^{2} ~ds \\
&\leq& C^{2}_{n} \left( \int_{\mathcal{O}} \left| u_{1}(s) - u_{2}(s) \right|_{\mathcal{H}}^{6} ~ds \right)^{\frac{1}{3}} 
\left( \int_{\mathcal{O}} \left( \left| u_{1}(s) \right|_{\mathcal{H}}^{2n-2} + \left| u_{2}(s) \right|_{\mathcal{H}}^{2n-2} \right)^{3} ~ds \right)^{\frac{2}{3}} \notag \\
\end{eqnarray*}

The following inequality can be obtained using the Minkowski inequality:

\begin{align}
  \left|F_{4}(u_{1})-F_{4}(u_{2})\right|_{\mathcal{H}} &\leq C_{n}\| u_{1} - u_{2} \|_{L^{6}(\mathcal{O})}
     \left( \left( \int_{\mathcal{O}} \left|u_{1}(s)\right|_{\mathcal{H}}^{6n-6} ~ds \right)^{\frac{1}{3}} + \left( \int_{\mathcal{O}} \left|u_{2}(s)\right|_{\mathcal{H}}^{6n-6} ~ds \right)^{\frac{1}{3}} \right) 
\end{align}

We know that:

\begin{align*}
    \left|u_{}(s)\right|_{\mathcal{H}}^{6n-6} \leq \max \{1, |u(s)|_{\mathcal{H}}\}^{6n-6}
\end{align*}

Therefore,

\begin{align*}
    \left( \int_{\mathcal{O}} \left|u_{}(s)\right|_{\mathcal{H}}^{6n-6} ~ds \right)^{\frac{1}{3}} \leq \begin{cases}
\left( |\mathcal{O}| + |u(s)|_{\mathcal{L}^{6n}(\mathcal{O})}^2 \right)^{1/3}, & d = 2 \\
\left( |\mathcal{O}| + |u(s)|_{\mathcal{L}^{6}(\mathcal{O})}^2 \right)^{1/3}, & d = 3
\end{cases}
\end{align*}

Hence, for $d=2$ and $d=3$ using the continuity of embedding $\mathcal{V} \hookrightarrow \mathcal{L}^{6n}(\mathcal{O})$ and  $\mathcal{V} \hookrightarrow \mathcal{L}^{6}(\mathcal{O})$ respectively, we can deduce that there is a constant $K_{n}$ such that

\begin{align}{\label{f-4}}
    \left|F_{4}(u_{1})-F_{4}(u_{2})\right|_{\mathcal{H}}   &\leq   
   K_{n} 
     \left( 1+\|u_{1}\|_{\mathcal{V}}^{2} + \|u_{2}\|_{\mathcal{V}}^{2} \right)^{\frac{1}{3}} \| u_{1} - u_{2} \|_{\mathcal{V}}
\end{align}

From (\ref{F_{1}}), (\ref{F_{2}}), (\ref{F_{3}}), and (\ref{f-4}) it follows that

 \begin{eqnarray*}
&& \left|F(u_{1})-F(u_{2})\right|_{\mathcal{H}}  
 \leq  \left|F_{1}(u_{1})-F_{1}(u_{2})\right|_{\mathcal{H}} + \left|F_{2}(u_{1})-F_{2}(u_{2})\right|_{\mathcal{H}} +\left|F_{3}(u_{1})-F_{3}(u_{2})\right|_{\mathcal{H}} 
    \\ && +\left|F_{4}(u_{1})-F_{4}(u_{2})\right|_{\mathcal{H}}  \\  
      & \leq& 2 C \left( \|  u_{1}\|^{2}_{\mathcal{V}} + \|  u_{2}\|^{2}_{\mathcal{V}}   +  \|  u_{1}\|_{\mathcal{V}} \|  u_{2}\|_{\mathcal{V}} \right)+ C_{n}\begin{bmatrix}
\left( \frac{2n-1}{2}\right) \left( \|  u_{1}\|^{2n-1}_{\mathcal{V}} +\|  u_{2}\|^{2n-1}_{\mathcal{V}} \right)\left(\|u_{1}\|_{\mathcal{V}}  
 +  \|u_{2}\|_{\mathcal{V}}\right) \\+ \left(\|  u_{1}\|^{2n}_{\mathcal{V}} + \|  u_{2}\|^{2n}_{\mathcal{V}}  \right)+ \left(  1+\|u_{1}\|_{\mathcal{V}}^{2} +    \|u_{2}\|_{\mathcal{V}}^{2}   \right)^{\frac{1}{3}}
\end{bmatrix}\\ &\times&
\|u_{1}-u_{2}\|_{\mathcal{V}}
 \end{eqnarray*}

Where $C= C_{1}C_{2}^{3}+2C_{3} C_{4}^{3}$ and $ C_{n} = \max\{K_{n}, C^{n+1}\} $  

From the above inequality, it follows that $F$ is locally Lipschitz and 

\begin{equation}
\left|F\left(u_1\right)-F\left(u_2\right)\right|_{\mathcal{H}} \leq \mathcal{G}\left(\left\|u_1\right\|_{\mathcal{V}},\left\|u_2\right\|_{\mathcal{V}}\right)\left\|u_1-u_2\right\|_{\mathcal{V}}, \quad u_1, u_2 \in \mathcal{V}
\end{equation}

\end{proof}

\begin{lemma}{\label{Self-Ad_op}} 
    The operator $A: D(A) \longrightarrow \mathcal{V}$ defined by 
\begin{align*}
    A &= \Delta^{2}-2\Delta 
\end{align*}
 is the self-adjoint operator.\\
\begin{proof}
    Suppose $u_{1}$ and $u_{2}$ are the elements in $D(A)$ then 
\begin{align*}
\langle Au_{1}, u_{2}  \rangle 
        &= \langle \Delta^{2} u_{1}-2\Delta u_{1} , u_{2}  \rangle = \langle \Delta^{2}u_{1} , u_{2}  \rangle - 2 \langle \Delta u_{1} , u_{2}  \rangle \\
        &=  \langle \Delta u_{1} , \Delta u_{2}  \rangle + 2 \langle \nabla u_{1} , \nabla u_{2}  \rangle =\langle  u_{1} , \Delta ^{2} u_{2}  \rangle - 2 \langle  u_{1} , \Delta u_{2}  \rangle \\&= \langle  u_{1} , \left(\Delta ^{2} -2 \Delta \right)u_{2} \rangle  = \langle u_{1}, Au_{2}  \rangle
\end{align*}  

Thus, operator $A$ is the self-adjoint operator.
\end{proof}

\end{lemma}

\begin{theorem}

Let  $\mathcal{E} \subset\mathcal{V} \subset \mathcal{H}$ satisfy the assumption (\ref{Ass_2.2.2}) and for any given $K>0$ there is $T'(K)$ such that $\|u_{0}\| \leq K$, where $u_{0} \in \mathcal{V}$, there is the unique local solution $ u : [0, T') \longrightarrow \mathcal{V}$ to the problem

\begin{align}{\label{PB}}
        \frac{\partial u}{\partial t} &=\pi _{u}(-\Delta^{2}u+2\Delta u +au + u^{2n-1}) = -Au + F(u(t)) \notag\\ 
u(0) &=u_{0}.  
    \end{align}
Where   ~~~~   $ F(u)=\|  u\|^{2}_{{H}^{2}_{0}} ~u + 2\|    u\|^{2}_{{H}^{1}_{0}} ~u  +\| u\|^{2n}_{{L}^{2n}} u- u^{2n-1}  $  ~~~    and  ~~~    $ A = \Delta^{2}-2\Delta$
\end{theorem} 

\begin{proof}
    A is a self-adjoint operator by (\ref{Self-Ad_op}), and by using lemma (\ref{lemma_abt_liptz}), $F(u)$ is locally Lipschitz. Therefore, by referring to (\ref{Pro.2}), (\ref{zorn}), and (\ref{proposition for unique local solution}), we can deduce that problem (\ref{PB}) has a unique maximal local solution in $\mathcal{V}$.
\end{proof}

\section{In-variance of Manifold and Global solution to the main problem}
Before proving the invariance of the Hilbert manifold, we mention some crucial lemmas. 
\begin{lemma}{\label{Lemma_weak derivative_Abs}}
    \cite{temam2000navier}, \cite{hussain2015analysis} Assume that $\mathcal{E}, ~~\mathcal{V} , ~~ \mathcal{H}$ satisfy the assumption (\ref{Ass_2.2.2}) then the dual  $\mathcal{H'}$ of $\mathcal{H}$ and dual $\mathcal{V'}$ of $\mathcal{V}$ satisfy the following relation 
\begin{align*}
    \mathcal{V} \hookrightarrow \mathcal{H} \cong \mathcal{H'} \hookrightarrow \mathcal{V'}
\end{align*}
Additionally, if the abstract function $u$ is in $L^{2}\left ( 0, T; \mathcal{V}\right)$ then its weak derivative $\frac{\partial u}{ \partial t}$   is in $L^{2}\left ( 0, T; \mathcal{V'}\right)$. \\
Furthermore, 
\begin{align*}
    |u(t)|^{2} = |u_{0}|^{2}+ 2 \int_{0}^{t}{ \langle u'(p), u(p) \rangle}~dp,  ~~~~~~~t \in [0,T]~~~~~~~- a.e
\end{align*} 

\end{lemma}
\begin{lemma}{\label{Lemma_weak derivative}}
    If $u(t) \in  X_{T}:=L^{2}\left( 0,T;\mathcal{E}\right) \cap C\left( \left[ 0,T\right] ;\mathcal{V}\right)$ and $u$ is a solution to the  main problem (\ref{main_Prb}), then the weak derivative is in $L^{2}\left ( 0, T; \mathcal{V}\right) $, that is, $\frac{\partial u}{ \partial t} \in L^{2}\left ( 0, T; \mathcal{H}\right)$. \\ 

    \begin{proof}
        
We have the following relation from the main problem (\ref{main_Prb}):
\begin{align*}
        \frac{\partial u}{\partial t} &= -Au + F(u(t)) 
    \end{align*}
Where   ~~~~   $ F(u)=\|  u\|^{2}_{{H}^{2}_{0}} ~u + 2\|    u\|^{2}_{{H}^{1}_{0}} ~u  +\| u\|^{2n}_{{L}^{2n}} u- u^{2n-1}  $  ~~~    and  ~~~    $ A = \Delta^{2}-2\Delta$\\  
For showing  $\frac{\partial u}{ \partial t} \in L^{2}\left ( 0, T; \mathcal{H}\right)$, we have to show that every term from the right side of the above equation is in $L^{2}\left ( 0, T; \mathcal{H}\right)$\\
We set; \\
\begin{align*}
    \frac{\partial u}{ \partial t} = \sum_{i=1}^{5}a_{i}
\end{align*}
It follows that:\\ 
$$a_{1}:=-Au, \quad
    a_{2}:=\|  u\|^{2}_{{H}^{2}_{0}} u,\quad
    a_{3}:=2\|    u\|^{2}_{{H}^{1}_{0}} u,\quad
    a_{4}:=\| u\|^{2n}_{{L}^{2n}} u,\quad
    a_{5}:=- u^{2n-1}.$$
Examine the equation $ a_{1}=-Au $
\begin{align*}
    \int_{0}^{T}{|a_{1}(p)|^{2}_{\mathcal{H}}}~dp &= \int_{0}^{T}{|Au(p)|^{2}_{\mathcal{H}}}dp = \int_{0}^{T}{|u(p)|^{2}_{D(A)}}dp \leq \sup_{t \in [0,T]}{ \|u(t)\|^{2}_{\mathcal{V}}} + \int_{0}^{T}{|u(p)|^{2}_{D(A)}}dp  \\
    & =|u(t)|^{2}_{X_{t}}< \infty.
\end{align*} 
It follows that $ a_{1} \in L^{2}\left ( 0, T; \mathcal{H}\right) $ \\

 Next, we examine the second function $ a_{2} =\|  u\|^{2}_{{H}^{2}_{0}} ~u.$ \\
 $ \int_{0}^{T}{|a_{2}(p)|^{2}_{\mathcal{H}}}dp$
 \begin{align*}
      &=  \int_{0}^{T}{| \|  u\|^{2}_{{H}^{2}_{0}} ~u|^{2}_{\mathcal{H}}}dp = \sup_{t \in [0,T]}{ \|u(t)\|^{4}_{{H_{0}^{2}}}} \int^{0}_{T}{|u(p)|_{\mathcal{H}}}~dp \\
      & \leq \sup_{t \in [0,T]}{ \|u(t)\|^{4}_{{H_{0}^{2}}}} \int^{0}_{T}{K_{1}|u(p)|_{\mathcal{V}}}dp \leq K_{1}K_{2} \sup_{t \in [0,T]}{ \|u(t)\|^{4}_{{\mathcal{V}}}} \sup_{t \in [0,T]}{ \|u(t)\|^{4}_{{\mathcal{V}}}} \\
      & \leq K_{1}K_{2} T \left( \sup_{t \in [0,T]}{ \|u(t)\|^{2}_{{\mathcal{V}}}} \right) ^{3} \leq K_{1}K_{2} T\left( \sup_{t \in [0,T]}{ \|u(t)\|^{2}_{\mathcal{V}}} + \int_{0}^{T}{|u(p)|^{2}_{D(A)}}~dp \right) ^{3} \\
      & \leq K_{1}K_{2} T |u(t)|^{6}_{X_{t}}< \infty.
      \end{align*} 
It follows that $ a_{2} \in L^{2}\left ( 0, T; \mathcal{H}\right) $ \\

Now consider the examine function $ a_{3} =2\|    u\|^{2}_{{H}^{1}_{0}} ~u, $ \\
$\int_{0}^{T}{|a_{3}(p)|^{2}_{\mathcal{H}}}dp$
 \begin{align*}
       &=  \int_{0}^{T}{| 2\|  u\|^{2}_{{H}^{1}_{0}} ~u|^{2}_{\mathcal{H}}}dp = 2 \sup_{t \in [0,T]}{ \|u(t)\|^{4}_{{H_{0}^{1}}}} \int^{0}_{T}{|u(p)|_{\mathcal{H}}}~dp \\
      &\leq 2 \sup_{t \in [0,T]}{ \|u(t)\|^{4}_{{H_{0}^{1}}}} \int^{0}_{T}{K_{1}|u(p)|_{\mathcal{V}}}dp\leq 2K_{1}K_{2} \sup_{t \in [0,T]}{ \|u(t)\|^{4}_{{\mathcal{V}}}} \sup_{t \in [0,T]}{ \|u(t)\|^{4}_{{\mathcal{V}}}} \\& \leq 2K_{1}K_{2} T \left( \sup_{t \in [0,T]}{ \|u(t)\|^{2}_{{\mathcal{V}}}} \right) ^{3} \leq 2K_{1}K_{2} T\left( \sup_{t \in [0,T]}{ \|u(t)\|^{2}_{\mathcal{V}}} + \int_{0}^{T}{|u(p)|^{2}_{D(A)}}~dp \right) ^{3} \\
      & \leq2 K_{1}K_{2} T |u(t)|^{6}_{X_{t}} < \infty.
      \end{align*}   

It follows that $ a_{3} \in L^{2}\left ( 0, T; \mathcal{H}\right) $ \\ 

Again examine the function $ a_{4}=\| u\|^{2n}_{{L}^{2n}} u, $\\
$\int_{0}^{T}{|a_{4}(p)|^{2}_{\mathcal{H}}}~dp$
\begin{align*}
     &=  \int_{0}^{T}{| \|  u\|^{2n}_{L^{2n}} u(p)|^{2}_{\mathcal{H}}}dp = \int_{0}^{T}{ \|  u\|^{4n}_{L^{2n}} |u(p)|^{2}_{\mathcal{H}}}dp\leq K^{4n} \int_{0}^{T}{ \|  u\|^{4n}_{\mathcal{V}} ~|u(p)|^{2}_{\mathcal{H}}}~dp \\
    & \leq K^{4n+2} \sup_{t \in [0,T]}{ \|  u\|^{4n}_{\mathcal{V}}}\int_{0}^{T}{ ~|u(p)|^{2}_{\mathcal{V}}}~dp \leq K^{4n+2} T \sup_{t \in [0,T]}{ \|  u\|^{4n+2}_{\mathcal{V}}} \\
    & \leq K^{4n+2} T |u(t)|^{4n+2}_{X_{t}}< \infty.
\end{align*}
      It follows that $ a_{4} \in L^{2}\left ( 0, T; \mathcal{H}\right) $ \\ 
Finally, check out for the function $  a_{5}=- u^{2n-1},$\\

$ \int_{0}^{T}{|a_{5}(p)|^{2}_{\mathcal{H}}}~dp$
\begin{align*}
      &=  \int_{0}^{T}{|  ~u^{2n-1}(p)|^{2}_{\mathcal{H}}}dp\leq K^{4n-2} \int_{0}^{T}{\|  ~u^{2n-1}(p)\|^{2}_{\mathcal{V}}}dp\leq K^{4n-2} T \left( \sup_{t \in [0,T]}{\|u(t)\|^{2}}_{\mathcal{V}}\right)^{2n-1} \\& \leq  K^{4n-2} T |u(t)|^{4n-2}_{X_{t}} < \infty
\end{align*} 
It follows that $ a_{5} \in L^{2}\left ( 0, T; \mathcal{H}\right) $ \\ 
Each $a_{i}$ belongs to $ L^{2}\left ( 0, T; \mathcal{H}\right)$. Therefore, we deduce that: 
\begin{align*}
    \frac{\partial u}{ \partial t} = \sum_{i=1}^{5}a_{i} \in L^{2}\left ( 0, T; \mathcal{H}\right)
\end{align*} 
We now introduce a fundamental proposition that is key to the existence of a global solution to our main problem.  

\end{proof}
\end{lemma}

\begin{proposition}
If $u(t)$, where $t \in [0,T)$,  is a solution of (\ref{main_Prb}) with  $|u_{0}|_{\mathcal{H}}^{2}=1$, $u(t) \in \mathcal{M}, ~~~~~\forall t \in [0,T)$; that is,
\begin{align*}
    u(t) \in \mathcal{M}, ~~~ \forall t \in [0, T)~~~i.e ~~~|u(t)|_{\mathcal{H}}^{2}=1, ~~~ \forall t \in [0, T)
\end{align*} 

\begin{proof}

By using lemmas ( \ref{Lemma_weak derivative_Abs}) and ( \ref{Lemma_weak derivative}), it follows that 
\begin{align*}
    \frac{1}{2} \left( |u(t)|_{\mathcal{H}}^{2}-1\right) &=\frac{1}{2} \left( |u_{0}(t)|_{\mathcal{H}}^{2}-1\right) + \int^{t}_{0}{\langle u'(p), u(p) \rangle} ~dp \\ 
    &= \frac{1}{2} \left( |u_{0}(t)|_{\mathcal{H}}^{2}-1\right) + \int^{t}_{0}{\langle -\Delta^{2} u(p)+2\Delta u(p) , u(p)  \rangle }\\
    &+ \int^{t}_{0}{ \langle \|  u\|^{2}_{{H}^{2}_{0}} ~u(p) + 2\|    u\|^{2}_{{H}^{1}_{0}} ~u(p)  +\| u\|^{2n}_{{L}^{2n}} u(p)- u^{2n-1}(p), u(p) \rangle} ~dp \\ 
    &= \frac{1}{2} \left( |u_{0}(t)|_{\mathcal{H}}^{2}-1\right) + \int^{t}_{0}{ \langle -\Delta^{2} u(p) ,u(p) \rangle} ~dp+  \int^{t}_{0}{ \langle 2\Delta u(p) ,u(p) \rangle} ~dp \\
    &+ \int^{t}_{0}{ \langle\|  u\|^{2}_{{H}^{2}_{0}} ~u(p) ,u(p) \rangle} ~dp + \int^{t}_{0}{ \langle 2\|    u\|^{2}_{{H}^{1}_{0}} ~u(p) ,u(p) \rangle} ~dp+ \int^{t}_{0}{ \langle \| u\|^{2n}_{{L}^{2n}} u(p) ,u(p) \rangle} ~dp\\
    &+ \int^{t}_{0}{ \langle -u^{2n-1}(p) ,u(p) \rangle} ~dp 
\end{align*} 
As ~~~~$ \frac{\partial u}{\partial t} =\pi _{u}(-\Delta^{2}u+2\Delta u -au - u^{2n-1})  \in \mathcal{M}$ ~~~ so~~~ $ \langle \frac{\partial u(p)}{\partial t}, u(p) \rangle = 0$~~~ \\
In addition, ~~~ $ \langle u(p) , u(p)\rangle = |u(p)|^{2}_{\mathcal{H}}$,  and $ |u_{0}|^{2}_{\mathcal{H}}= 1$ It follows that 

\begin{align*}
     \frac{1}{2} \left( |u(t)|_{\mathcal{H}}^{2}-1\right) &= \int_{0}^{t} {\left(\|u\|^{2}_{H^{2}_{0}} + 2 \|u\|^{2}_{H^{1}_{0}} + \|u\|^{2n}_{L^{2n}}\right) \left( |u(p)|^{2}_{\mathcal{H}}-1\right)}~dp
\end{align*} 

Let $ \psi (t) = \left( |u(t)|^{2}_{\mathcal{H}}-1\right) $, implying that

\begin{align*}
     \frac{1}{2} \psi (t) &= \int_{0}^{t} {\left(\|u\|^{2}_{H^{2}_{0}} + 2 \|u\|^{2}_{H^{1}_{0}} + \|u\|^{2n}_{L^{2n}}\right) \psi (p)}~dp
\end{align*} 

\begin{align*}
      \psi '(t) &= 2{\left(\|u\|^{2}_{H^{2}_{0}} + 2 \|u\|^{2}_{H^{1}_{0}} + \|u\|^{2n}_{L^{2n}}\right) \psi (t)}
\end{align*} 
\begin{align*}
      \psi (t) &= C e^{2{\left(\|u\|^{2}_{H^{2}_{0}} + 2 \|u\|^{2}_{H^{1}_{0}} + \|u\|^{2n}_{L^{2n}}\right)}t}   
\end{align*} 
Now, $\psi (0) =C =0  \implies \psi (t)=0 $. It follows that
\begin{align*}
    |u(t)|^{2}_{\mathcal{H}}-1 = 0 \\
    |u(t)|^{2}_{\mathcal{H}}= 1
\end{align*} 
thus, it is proved that $ u(t) \in \mathcal{M}$
\end{proof} 
\end{proposition}

The following lemma introduces the energy function:

\begin{lemma} {\label{Energy_ftn}}
    The map $ \mathcal{Y}: \mathcal{V}\longrightarrow R $ defined as 
\begin{align*}
    \mathcal{Y}(u) = \frac{1}{2} \|u\|^{2}_{\mathcal{V}} + \frac{1}{2n} \|u\|^{2n}_{L^{2n}}, ~~~~ n \in N 
\end{align*} 
belongs to the $C^{2}$ class.
\end{lemma}

\begin{theorem}{\label{global_sol_thm}}

By assuming $u$ as the local mild solution of the problem (\ref{main_Prb}), for $t \in [0,T)$ with initial condition  $u_{0} \in  \mathcal{V} \cup \mathcal{M}$, we have

\begin{align*}
    \|u\|_{\mathcal{V}} \leq 2 ~\mathcal{Y}(u_{0}), ~~~~ \forall t \in [0,T)
\end{align*} 
where $\mathcal{Y}$ denotes the energy function defined in (\ref{Energy_ftn}). Furthermore, $T= \infty$. More precisely, $u$ is a unique global solution to problem (\ref{main_Prb}).\\

\begin{proof}
    By definition of the norm on $\mathcal{V}$ and using lemma ( \ref{Lemma_weak derivative_Abs}) and ( \ref{Lemma_weak derivative}), the following relation is considered:\\
    $\frac{\|u\|^{2}_{\mathcal{V}}}{2}-  \frac{\|u_{0}\|^{2}_{\mathcal{V}}}{2}$
\begin{align*}
     &= \frac{\|u\|^{2}_{L^{2}}}{2}+   {\|\nabla u\|^{2}_{L^{2}}}{} + \frac{\|\Delta u\|^{2}_{L^{2}}}{2}+ \frac{\| u\|^{2}_{H^{1}_{0}}}{2}-\frac{\|u_{0}\|^{2}_{L^{2}}}{2}-   {\|\nabla u_{0}\|^{2}_{L^{2}}}{} - \frac{\|\Delta u_{0}\|^{2}_{L^{2}}}{2}- \frac{\| u_{0}\|^{2}_{H^{1}_{0}}}{2}\\
     &=  2 \int^{t}_{0}{ \langle \nabla u(p), \nabla u_{p}(p) \rangle}~dp + \int^{t}_{0}{ \langle \Delta u(p), \Delta u_{p}(p) \rangle}~dp +  \int^{t}_{0}{ \langle u(p),  u_{p}(p) \rangle}~dp
\end{align*}  

As $ u(t) \in \mathcal{V} \cap \mathcal{M}$, $ u_{p}(p) \in T\mathcal{M}$ and $\langle u(p),  u_{p}(p) \rangle = 0 $, it follows that\\
$\frac{\|u\|^{2}_{\mathcal{V}}}{2}-  \frac{\|u_{0}\|^{2}_{\mathcal{V}}}{2}$
\begin{align*}
     &= 2  \int^{t}_{0}{ \langle \nabla u(p), \nabla u_{p}(p) \rangle}~dp + \int^{t}_{0}{ \langle \Delta u(p), \Delta u_{p}(p) \rangle}~dp  \\
    &=      - 2\int^{t}_{0}{ \langle \Delta u(p),  u_{p}(p) \rangle}~dp + \int^{t}_{0}{ \langle \Delta^{2} u(p),  u_{p}(p) \rangle}~dp  \\
    &= \int^{t}_{0} {\langle \Delta^{2} u(p) -2 \Delta u(p),  u_{p}(p) \rangle}~dp \\
    &= \int^{t}_{0} {\langle \Delta^{2} u(p) -2 \Delta u(p) + u_{p}(p),  u_{p}(p) \rangle}~dp  - \int^{t}_{0} {\langle u_{p}(p), u_{p}(p) \rangle } ~ dp \\
    &= -\int^{t}_{0} {\left|\frac{du}{dp} \right|^{2}} ~dp + \int^{t}_{0} {\langle \|  u\|^{2}_{{H}^{2}_{0}} ~u(p) + 2\|    u\|^{2}_{{H}^{1}_{0}} ~u (p) +\| u\|^{2n}_{{L}^{2n}} u(p)- u(p)^{2n-1} ,  u_{p}(p) \rangle}~dp
\end{align*}

As $ u(t) \in \mathcal{V} \cap \mathcal{M}$ so $ u_{p}(p) \in T\mathcal{M}$ and $\langle u(p),  u_{p}(p) \rangle = 0 $, it follows:

\begin{align*}
    \frac{\|u\|^{2}_{\mathcal{V}}}{2}-  \frac{\|u_{0}\|^{2}_{\mathcal{V}}}{2} &= -\int^{t}_{0} {\left\|u_{p}(p) \right\|_{L^{2}}^{2}} ~dp - \int_{0}^{t}{\langle u(p)^{2n-1} ,  u_{p}(p) \rangle}~dp \\
    &= -\int^{t}_{0} {\left\|u_{p}(p) \right\|_{L^{2}}^{2}} ~dp -\frac{1}{2n} \int_{0}^{t}{\int_{D}{\frac{d}{dp}u^{2n}(p)}}~dp\\
    &= -\int^{t}_{0} {\left\|u_{p}(p) \right\|_{L^{2}}^{2}} ~dp -\frac{1}{2n} \int_{0}^{t}{{\frac{d}{dp} \|u(p)\|^{2n}_{L^{2n}}}}~dp\\
    &= -\int^{t}_{0} {\left\|u_{p}(p) \right\|_{L^{2}}^{2}} ~dp - \left( \frac{\|u\|^{2n}_{L^{2n}}}{2n}-\frac{\|u_{0}\|^{2n}_{L^{2n}}}{2n}\right) \\
     \frac{\|u\|^{2}_{\mathcal{V}}}{2}+ \frac{1}{2n} \|u\|^{2n}_{L^{2n}}- \left( \frac{\|u_{0}\|^{2}_{\mathcal{V}}}{2}+\frac{1}{2n}\|u_{0}\|^{2n}_{L^{2n}} \right) &= -\int^{t}_{0} {\left\|u_{p}(p) \right\|_{L^{2}}^{2}} ~dp\\
\mathcal{Y}(u(t))- \mathcal{Y}(u_{0}) &= -\int^{t}_{0} {\left\|u_{p}(p) \right\|_{L^{2}}^{2}} ~dp
\end{align*}
Where 
\begin{align*}
    \mathcal{Y}(u) = \frac{1}{2} \|u\|^{2}_{\mathcal{V}} + \frac{1}{2n} \|u\|^{2n}_{L^{2n}}, ~~~~ n \in N 
\end{align*} 

From the above relation, it is clear that $\mathcal{Y}$ is a non-increasing function and
\begin{align*}
    \mathcal{Y}(u(t))\leq  \mathcal{Y}(u_{0}),~~~\forall t \in [0, T)
\end{align*}
Therefore, we can deduce that:
\begin{align*}
   \|u(t)\|_{\mathcal{V}}\leq 2~\mathcal{Y}(u(t))\leq  2~\mathcal{Y}(u_{0}), ~~~~~~\forall t \in [0,T) \\
\sup_{t \in [0,T)}{\|u(t)\|_{\mathcal{V}}}\leq  2~\mathcal{Y}(u_{0}), ~~~~~~\forall t \in [0,T)
\end{align*}

Thus, taking $2~\mathcal{Y}(u_{0})=K$ and referring to Proposition (\ref{proposition_for_global_solution}), $T$ becomes infinity, that is,  $T= \infty$ and $u$ is a global solution to the main problem (\ref{main_Prb}).\\
\\
\textbf{Uniqueness of the global solution} 
Let us assume that $u(t):=u_{1}(t)-u_{2(t)}$ where\\
\begin{align*}
   & u_{1}(t)~~~ \text{satisfies}~~~  \frac{\partial u_{1}(t)}{\partial t} =  -\Delta^{2}u_{1}(t) + 2 \Delta u_{1}(t)+ F(u_{1}(t)) \\
    &\text{and}~~~ u_{2}(t) ~~~\text{satisfies}~~~ \frac{\partial u_{2}(t)}{\partial t} =  -\Delta^{2} u_{2}(t) + 2 \Delta u_{2}(t) + F(u_{2}(t)) \\
    &\text{such}\,\,\text{that} ~~~ u(0)= u_{1}(0)-u_{2}(0) = 0 ~~ \text{and} \\
&\frac{\partial u(t)}{\partial t}= \frac{\partial }{\partial t}\left(u_{1}(t)-u_{2}(t) \right)  =  -\Delta^{2} u(t) + 2 \Delta u(t) + F(u_{1}(t))-F(u_{2}(t))
\end{align*}
As shown previously,
\begin{align*}
     \frac{\|u\|^{2}_{\mathcal{V}}}{2}-  \frac{\|u_{0}\|^{2}_{\mathcal{V}}}{2} &= -\int^{t}_{0} {\left\|u_{p}(p) \right\|_{L^{2}}^{2}} ~dp+ \int^{t}_{0} {\langle \Delta^{2} u(p) -2 \Delta u(p) + u_{p}(p),  u_{p}(p) \rangle}~dp   \\    
\end{align*}
By using Cauchy Schwartz Inequality, it follows:\\
$ \frac{\|u\|^{2}_{\mathcal{V}}}{2}-  \frac{\|u_{0}\|^{2}_{\mathcal{V}}}{2}$
\begin{align*}
    &\leq   -\int^{t}_{0} {\left\|u_{p}(p) \right\|_{L^{2}}^{2}} ~dp + \int^{t}_{0} \left( \int_{D}|\Delta^{2} u(p) -2 \Delta u(p) + u_{p}(p)|^{2} ~dp\right)^{\frac{1}{2}}{\left(\int_{D}|u_{p}(p)|^{2} ~dp\right)^{\frac{1}{2}}   }~dp \\
    & \leq  -\int^{t}_{0} {\left\|u_{p}(p) \right\|_{L^{2}}^{2}} ~dp + \int^{t}_{0} {\|\Delta^{2} u(p) -2 \Delta u(p) + u_{p}(p)\|_{L^{2}}}{\|u_{p}(p)\|_{L^{2}}} ~dp  
\end{align*}

Using the fact that, $ ab \leq \frac{a^{2}+b^{2}}{2}$, it follows:
\begin{align*}
     \frac{\|u\|^{2}_{\mathcal{V}}}{2}-  \frac{\|u_{0}\|^{2}_{\mathcal{V}}}{2} &\leq   -\int^{t}_{0} {\left\|u_{p}(p) \right\|_{L^{2}}^{2}} ~dp + \frac{1}{2}\int^{t}_{0} {\|\Delta^{2} u(p) -2 \Delta u(p) + u_{p}(p)\|^{2}_{L^{2}}}~dp + \frac{1}{2} \int_{0}^{t}{\|u_{p}(p)\|^{2}_{L^{2}}} ~dp 
\end{align*}

    As $u_{0}(0)=0$ and $\|u_{0}\|_{\mathcal{V}} =0 $, it follows that
    \begin{align*}
       \frac{\|u\|^{2}_{\mathcal{V}}}{2} &\leq   -\int^{t}_{0} {\left\|u_{p}(p) \right\|_{L^{2}}^{2}} ~dp + \frac{1}{2}\int^{t}_{0} {\|\Delta^{2} u(p) -2 \Delta u(p) + u_{p}(p)\|^{2}_{L^{2}}}dp\\ &+ \frac{1}{2} \int_{0}^{t}{\|u_{p}(p)\|^{2}_{L^{2}}} dp  \\
       & \leq - \frac{1}{2}\int^{t}_{0} {\left\|u_{p}(p) \right\|_{L^{2}}^{2}} ~dp + \frac{1}{2}\int^{t}_{0} {\|\Delta^{2} u(p) -2 \Delta u(p) + u_{p}(p)\|^{2}_{L^{2}}}~dp\\
      \|u\|^{2}_{\mathcal{V}} +  \int^{t}_{0} {\left\|u_{p}(p) \right\|_{L^{2}}^{2}} ~dp &\leq  \int^{t}_{0} {\|\Delta^{2} u(p) -2 \Delta u(p) + u_{p}(p)\|^{2}_{L^{2}}}~dp \\
      & \leq \int_{0}^{t}{\|F(u_{1})-F(u_{2})\|^{2}}~dp
    \end{align*} 

But, $\|F(u_{1})-F(u_{2})\| \leq G\left(\|u_{1}\|,\|u_{2}\|\right)\|u_{1}-u_{2}\|_{\mathcal{V}{}}$, it follows:
\begin{align*}
     \|u\|^{2}_{\mathcal{V}} +  \int^{t}_{0} {\left\|u_{p}(p) \right\|_{L^{2}}^{2}} ~dp &\leq  \int_{0}^{t}{\left(G\left(\|u_{1}\|,\|u_{2}\|\right)\|\right)^{2}}\|u_{1}-u_{2}\|_{\mathcal{V}}^{2}~dp\\
     & \leq   \int_{0}^{t}{\left(G\left(\|u_{1}\|,\|u_{2}\|\right)\|\right)^{2}}\|u\|_{\mathcal{V}}^{2}~dp
\end{align*}
Using Gronwall's inequality, we obtain:

\begin{align*}
    \|u(t)\|^{2}_{\mathcal{V}} \leq 0 \cdot.
\end{align*}
So it follows that,
\begin{equation}
  u(t)=0 \quad\text{ i.e. }\quad 
    u_{1}(t)=u_{2}(t).  
\end{equation}

Thus, the uniqueness of the global solution to problem (\ref{main_Prb}) was proved.
\end{proof}  
\end{theorem}

\section{Gradient System} 
In this section, we demonstrate that the global solution to the main problem (\ref{main_Prb}) forms a gradient system.
Before demonstrating this, we discuss the important prepositions and definitions related to them.

\begin{proposition}{\label{Pre_compact_lemma}}

Let the solution $u(t)$ follow theorem (\ref{global_sol_thm}); then, the orbit $\{ u(t); t\geq 1\}$ is pre-compact in $\mathcal{V}$.
\begin{proof}
    Recall that $\{ u(t); t\geq 1\}$  is precompact in $\mathcal{V}$ if it is bounded in $D(A^{\mu})$.\\ 
    
    What we are about to do is that  $\{ u(t); t\geq 1\}$ is bounded in $D(A^{\mu})$, ~~for $\mu > \frac{1}{2}$  \\ \text{Where,}~~~~$A= \Delta^{2}-2 \Delta$ \\
    \\
From the variation in the constant formula, we have:
\begin{align*}
    A^{\mu}u(t)= A^{\mu} e^{-At}u_{0}+ \int^{t}_{0} A^{\mu} e^{-A(t-p)} F(u(p)) ~dp
\end{align*}
Where $ F(u)=\|  u\|^{2}_{{H}^{2}_{0}} ~u + 2\|    u\|^{2}_{{H}^{1}_{0}} ~u  +\| u\|^{2n}_{{L}^{2n}} u- u^{2n-1}  $

\begin{align*}
\left|A^{\mu}u(t) \right|_{\mathcal{H}} & = \left|  A^{\mu} e^{-At}u_{0}+ \int^{t}_{0} A^{\mu} e^{-A(t-p)} F(u(p)) ~dp\right|_{\mathcal{H}}\\
 & \leq \left| A^{\mu} e^{-At}u_{0}\right|_{\mathcal{H}} + \left|\int^{t}_{0} A^{\mu} e^{-A(t-p)} F(u(p)) ~dp \right|_{\mathcal{H}}\\
 & \leq  \left|A^{\mu} e^{-At}u_{0} \right|_{\mathcal{H}} + \int^{t}_{0} \left|A^{\mu} e^{-A(t-p)} F(u(p))\right|_{\mathcal{H}} ~dp
\end{align*}

As by using proposition 1.23 in \cite{henry2006geometric} , we have $ \left|A^{\mu} e^{-At} \right|_{\mathcal{H}} \leq M_{\mu}t^{-\mu} e^{-\delta t}$\\
therefore, 
\begin{align*}
    \left|A^{\mu}u(t) \right|_{\mathcal{H}} & \leq M_{\mu}t^{-\mu} e^{-\delta t} \left|u_{0}\right|_{\mathcal{H}}+ \int^{t}_{0} M_{\mu} t^{t-\mu}e^{-\delta(t-p)} \left|F(u(p))\right|_{\mathcal{H}} ~dp
\end{align*}
As $u_{0} \in \mathcal{M}$ so $ \left|u_{0}\right|_{\mathcal{H}}=1 $,
\begin{align}{\label{Main_exp_compact}}
     \left|A^{\mu}u(t) \right|_{\mathcal{H}} & \leq M_{\mu}t^{-\mu} e^{-\delta t} + \int^{t}_{0} M_{\mu} (t-p)^{-\mu}e^{-\delta(t-p)} \left|F(u(p))\right|_{\mathcal{H}} ~dp
\end{align}
Now we will compute $\left|F(u(p))\right|_{\mathcal{H}} $\\

\begin{align*}
    \left|F(u(p))\right|_{\mathcal{H}}& =\left| \|  u\|^{2}_{{H}^{2}_{0}} ~u + 2\|    u\|^{2}_{{H}^{1}_{0}} ~u  +\| u\|^{2n}_{{L}^{2n}} u- u^{2n-1} \right|_{\mathcal{H}} \\
    & \leq  \|  u\|^{2}_{{H}^{2}_{0}} ~|u|_{\mathcal{H}} + 2\|    u\|^{2}_{{H}^{1}_{0}} ~|u|_{\mathcal{H}} +\| u\|^{2n}_{{L}^{2n}} |u|_{\mathcal{H}}+ \left| u^{2n-1} \right|_{\mathcal{H}}
\end{align*}
 Because $u \in \mathcal{M}$, $ \left|u\right|_{\mathcal{H}}=1 $. And $\mathcal{V} \hookleftarrow {{H}^{2}_{0}} $, $\mathcal{V} \hookleftarrow {{H}^{1}_{0}} $ and $\mathcal{V} \hookleftarrow {L^{2n}} $ is therefore $  \|  u\|^{2}_{{H}^{2}_{0}}  \leq \|  u\|^{2}_{\mathcal{V}}$, and $  \|  u\|^{2}_{{H}^{1}_{0}}  \leq \|  u\|^{2}_{\mathcal{V}}$  it follows:
\begin{align*}
    \left|F(u(p))\right|_{\mathcal{H}}& \leq \|  u\|^{2}_{\mathcal{V}}  + 2\|    u\|^{2}_{\mathcal{V}}  +\| u\|^{2n}_{L^{2n}} + \left| u^{2n-1} \right|_{\mathcal{H}}
\end{align*} 

From the energy function $  \mathcal{Y}(u) = \frac{1}{2} \|u\|^{2}_{\mathcal{V}} + \frac{1}{2n} \|u\|^{2n}_{L^{2n}}, ~~~~ n \in N  $ we obtain

\begin{align*}
    \|  u\|^{2}_{\mathcal{V}} \leq 2 ~\mathcal{Y}(u_{0}) \\
    \text{and} \\
     \|  u\|^{2n}_{L^{2n}} \leq 2n ~\mathcal{Y}(u_{0})
\end{align*}
Therefore, 
\begin{align}{\label{F(u(p)}}
    \left|F(u(p))\right|_{\mathcal{H}}& \leq (4+2n) ~\mathcal{Y}(u_{0})+ \left| u^{2n-1} \right|_{\mathcal{H}}
\end{align}

Now consider the expression $\left| u^{2n-1} \right|_{\mathcal{H}}$ \\

\begin{align*}
    \left| u^{2n-1} \right|^{2}_{\mathcal{H}} & = \int_{D}{\left(u^{2n-1}(p)\right)^{2}} ~dp \\
    \left| u^{2n-1} \right|_{\mathcal{H}} & =\left( \int_{D}{u^{4n-2}(p) ~dp }\right)^{\frac{1}{2}}\\
    &= \left(\left( \int_{D}{u^{4n-2}(p) ~dp }\right)^{\frac{1}{4n-2}}\right)^{2n-1} \\
    &= \|u\|^{2n-1}_{L^{4n-2}}
\end{align*}
As $ \mathcal{V} \hookrightarrow L^{4n-2}$ so there is the constant  $K$ such that $ \|u\|^{2n-1}_{L^{4n-2}} \leq K^{2n-1} \|u\|^{2n-1}_{\mathcal{V}}$

\begin{align}{\label{embedig_4n-2}}
     \left| u^{2n-1} \right|_{\mathcal{H}} & \leq  K^{2n-1}\|u\|^{2n-1}_{\mathcal{V}}
\end{align}
By using (\ref{embedig_4n-2}), (\ref{F(u(p)}), and $\|u\|^{2n-1}_{\mathcal{V}} \leq 2^{2n-1} \left( \mathcal{Y}(u_{0})\right)^{2n-1}$, it follows that
 
\begin{align}{\label{F_final}}
    \left|F(u(p))\right|_{\mathcal{H}}& \leq (4+2n) ~\mathcal{Y}(u_{0})+ 2^{2n-1} K^{2n-1} \left( \mathcal{Y}(u_{0})\right)^{2n-1} : = N < \infty
\end{align}
Using (\ref{F_final}) and (\ref{Main_exp_compact}), it follows that 
\begin{align*}
    \left|A^{\mu}u(t) \right|_{\mathcal{H}} & \leq M_{\mu}t^{-\mu} e^{-\delta t} +M_{\mu}N \int^{t}_{0}  (t-p)^{-\mu}e^{-\delta(t-p)}  ~dp \\
\left|u(t)\right|_{D(A^{\mu)}} & \leq M_{\mu}t^{-\mu} e^{-\delta t} +M_{\mu}N \int^{\infty}_{0}  (t-p)^{-\mu}e^{-\delta(t-p)}  ~dp\\
     &\leq M_{\mu}t^{-\mu} e^{-\mu t} +M_{\mu}N ~\Gamma(1-\mu) := ~ Q_{\mu} < \infty
\end{align*}

Thus, $\{ u(t); t\geq 1\}$ is bounded $ D(A^{\mu})$, where $\mu > \frac{1}{2}$, and the orbit $\{ u(t); t\geq 1\}$ is pre compact in $\mathcal{V}$

\end{proof}  
\end{proposition} 

\begin{corollary}{\label{Coro-PreCompact}}
The $\Omega$- limit set ~ $\Omega(u_{0}) = \cap_{q \geq 1} \overline{\{ u(t); t \geq q\}}$ exists and is compact in $\mathcal{V}$.
\begin{proof}
   In preposition (\ref{Pre_compact_lemma}), we have shown that $\{ u(t); t \geq q\}$ is pre-compact in $\mathcal{V}$ for $q\geq 1$. As the closure of the pre-compact set is also a pre-compact set, $\overline{\{ u(t); t \geq q\}}$ is pre-compact in $\mathcal{V}$ for $q\geq 1$. \\
    Additionally, $\overline{\{ u(t); t \geq q\}}$ is closed, and thus, is complete in $\mathcal{V}$-norm. Hence, the completion and pre-compactness of $\overline{\{ u(t); t \geq q\}}$ implies that $\overline{\{ u(t); t \geq q\}}$ is compact $ \forall q\geq 1 \in \mathcal{V}$. \\
    Hence, $\Omega(u_{0})$ is the intersection of the decreasing nonempty compact sets and is not an empty set in $\mathcal{V}$.
\end{proof}
\end{corollary} 

\begin{definition}{\label{Lyponov_ftn}}
    \cite{hussain2015analysis}, \cite{zheng2004nonlinear} Assume the space $\mathcal{V}$ as a complete metric space and $\phi=\left(\phi_t\right)_{t\geq 0}$ is the non-linear $C_{0}$-semigroup on $\mathcal{V}$. The map $\mathcal{Y} :\mathcal{V}\rightarrow \mathbb{R}$, that is continuous,  is known as the \textbf{Lyapunov function} with respect to $\phi=\left(\phi_t\right)_{t\geq 0}$  if the following criteria are met:\\

i) $ \mathcal{Y} (\phi(t)x$ ~~~is a non-increasing monotone function in $t$, ~~~~~~~~~~, where $ \forall x \in \mathcal{V}$.\\

ii) The function $\mathcal{Y}$ is bounded below, that is, for any constant $K$, we have

\begin{align*}
    \mathcal{Y}(x) \geq K, ~~~~ ~~~~~~\forall x \in \mathcal{V}
\end{align*}

\end{definition}

\begin{definition}{\label{def_grad_flow}}
    \cite{hussain2015analysis}, \cite{zheng2004nonlinear} Assume the space $\mathcal{V}$ as a complete metric space and $\phi=\left(\phi_t\right)_{t\geq 0}$ is the non-linear $C_{0}$-semigroup on $\mathcal{V}$. Map $\mathcal{Y} :\mathcal{V}\rightarrow \mathbb{R}$ is the \textbf{Lyapunov function} with respect to $\phi=\left(\phi_t\right)_{t\geq 0}$. Then the system $(V,\phi,\mathcal{Y}  )$ is reffered as the \textbf{gradient system/flow} if the following criteria are met::\\

i) For an arbitrary  $x\in\mathcal{V}$, there exists $t_{0}>0$ so that ~~~$\bigcup _{t\geq t_{0}}\phi_tx$ is relatively compact in $\mathcal{V}$.\\

ii)  if 

\begin{align*}
    \forall t>0,\enspace \mathcal{Y}  \left( \phi_tx\right) =\mathcal{Y}  (x)
\end{align*}

then $x$ is the fixed point of semigroup $\phi_t$.

\end{definition}

\begin{theorem}
$(\mathcal{V},\phi,\mathcal{Y} )$ is a gradient system with $\mathcal{V}= D(A^{\frac{1}{2}})$,  $\phi=\left(\phi_t\right)_{t\geq 0}$ is a nonlinear $C_{0}$-semigroup, and $\mathcal{Y}$ is the energy function defined in Lemma (\ref{Energy_ftn}). 
\begin{proof}
    Function $  \mathcal{Y}(u) = \frac{1}{2} \|u\|^{2}_{\mathcal{V}} + \frac{1}{2n} \|u\|^{2n}_{L^{2n}}, ~~~~ n \in N  $ is a non-increasing function with $\mathcal{Y}(u) \geq 0 $. By Definition (\ref{Lyponov_ftn}), $\mathcal{Y}(u)$ is a Lyapunov function. To prove $(V,\phi,\mathcal{Y} )$ is a gradient system, we need to prove the conditions of definition (\ref{def_grad_flow}). Condition (i) of (\ref{def_grad_flow}) is proved in result (\ref{Coro-PreCompact}). Therefore, condition(ii) in (\ref{def_grad_flow}) remains to be verified. We now prove the above condition.\\
    As from the preposition (\ref{global_sol_thm}), we have 
    \begin{align*}
        \mathcal{Y}(u(t))- \mathcal{Y}(u_{0}) &= - \int^{t}_{0}{\left|\frac{du}{dp}\right|_{\mathcal{H}}^{2}} ~dp
    \end{align*} 
    If for any $t_{0} \geq 0$ we have
    \begin{align*}
        \mathcal{Y}(u(t_{0}))= \mathcal{Y}(\phi(u_{0}))= \mathcal{Y}(u_{0})
    \end{align*} 
    then 
    \begin{align*}
        \int^{t_{0}}_{0}{\left|\frac{du}{dp}\right|_{\mathcal{H}}^{2}} ~dp &= 0\\
        \iff \frac{du}{dp} &=0 \\
        \iff u(t)& = u(0) 
    \end{align*} 
  Thus, $u_{0}$ is a fixed point of $\phi$ and $(V,\phi,\mathcal{Y} )$ is the gradient system.
\end{proof} 

\end{theorem} 
\textbf{Declarations}\\ \\

\textbf{Conflict of interest:} We confirm that this manuscript is the author's original work, has not been published, and is not under consideration for publication elsewhere. We declare that the authors have no known competing financial interests or personal relationships that could have appeared to influence the work reported in this paper. The authors declare no conflict of interest.\\ \\
\textbf{Data Availability}\\\\\
No data were required to perform this research.\\\\
\textbf{Author Contributions:} All authors have contributed equally to this manuscript.\\ \\

\end{document}